\documentclass[a4paper,10pt]{article}
\usepackage[utf8]{inputenc}
\usepackage[english]{babel}
\usepackage{amsmath}
\usepackage{amsfonts}
\usepackage{amssymb}
\usepackage{amsthm}
\usepackage{graphicx}
\usepackage{a4wide}
\usepackage{xcolor}
\usepackage{cite}
\RequirePackage{multirow}
\RequirePackage{hyperref}
\RequirePackage{tikz}
\usetikzlibrary{arrows.meta}
\usepackage{pgfplots} 
\usepackage{pgfgantt}
\usepackage{pdflscape}
\usepackage{ifthen}
\pgfplotsset{compat=newest} 
\pgfplotsset{plot coordinates/math parser=false}

\theoremstyle{plain}						
\newtheorem{theorem}{Theorem}[section]

\theoremstyle{definition}

\theoremstyle{remark}
\newtheorem{remark}[theorem]{Remark}

\newlength\fwidth
\newlength\xmin
\newlength\xmax
\newlength\ymin
\newlength\ymax

\newcommand\RR{\mathbb{R}}
\newcommand\NN{\mathbb{N}}
\newcommand\hh{{\tfrac{h}{2}}}
\newcommand\half{{\frac12}}
\newcommand\jmo{{j-1}}
\newcommand\jmh{{j-\half}}
\newcommand\jpmh{{j\pm\half}}
\newcommand\jph{{j+\half}}
\newcommand\jpo{{j+1}}
\newcommand\npo{{n+1}}
\newcommand\dx{\mathit{dx}}
\newcommand\dy{\mathit{dy}}
\newcommand\BV{\mathit{BV}}
\renewcommand{\P}{P}
\newcommand{\w}{w}
\renewcommand{\c}{c}
\newcommand\opt{{\text{opt}}}
\newcommand\brho{{\bar\rho}}
\renewcommand{\S}{S}

\newcommand{\qn}{{\nu}}
\newcommand{\nqn}{{R}}
\newcommand{\qw}{{\gamma}}
\newcommand{\qwn}{\qw_\qn}
\newcommand{\qvar}{{y}}

\begin{document}

\title{Maximum principle satisfying CWENO schemes for non-local conservation laws}
\author{Jan Friedrich\footnotemark[1], \; Oliver Kolb\footnotemark[1]}

\footnotetext[1]{University of Mannheim, Department of Mathematics, 68131 Mannheim, Germany (janfriea@mail.uni-mannheim.de, kolb@uni-mannheim.de).}

\maketitle

\begin{abstract}
 Central WENO schemes are a natural candidate for higher-order schemes for non-local conservation laws, since the underlying reconstructions do not only provide single point values of the solution but a complete (high-order) reconstruction in every time step, which is beneficial to evaluate the integral terms.
Recently, in [C.~Chalons et al., \emph{SIAM J. Sci. Comput., 40(1), A288--A305}], Discontinuous Galerkin (DG) schemes and Finite Volume WENO (FV-WENO) schemes have been proposed to obtain high-order approximations for a certain class of non-local conservation laws. In contrast to their schemes, the presented CWENO approach neither requires a very restrictive CFL condition (as the DG methods) nor an additional reconstruction step (as the FV-WENO schemes). Further, by making use of the well-known linear scaling limiter of [X.~Zhang and C.-W.~Shu, \emph{J. Comput. Phys., 229, p.~3091--3120}], our CWENO schemes satisfy a maximum principle for suitable non-local conservation laws.
\end{abstract}



\section{Introduction}\label{sec:intro}
We are interested in the numerical solution of conservation laws with non-local flux function, i.e., the flux function does not only depend on the solution in a single point but on the solution in a (spatial) neighborhood of the considered point.
Such type of equations occur within several applications, e.g.\ to model supply chains \cite{armbruster2006continuum}, sedimentation \cite{betancourt2011nonlocal}, conveyor belts \cite{gottlich2014modeling}, crowd motion \cite{colombo2012class} or traffic flow \cite{blandin2016well, GoatinScialanga2016}.
The well-posedness of the Cauchy problem for certain non-local conservation laws has for instance been investigated in \cite{aggarwal2015nonlocal, amorim2015numerical, colombo2011control, blandin2016well, GoatinScialanga2016, ChiarelloGoatin2017}.

Within this work, we consider scalar non-local conservation laws of the form
\begin{equation}\label{eq:nonloc}
 \rho_t + \big( \underbrace{g(\rho) v(\rho * w_\eta)}_{= f(\rho)} \big)_x = 0, \quad x\in\RR, \quad t>0
\end{equation}
with given initial conditions
\begin{equation}\label{eq:initial}
 \rho(0,x) = \rho_0(x), \quad x\in\RR.
\end{equation}
In particular, referring to~\cite{ChalonsGoatinVillada2018}, we will consider a non-local traffic flow model given by
\begin{equation}\label{eq:trafficFlow}
 g(\rho)=\rho, \quad v(\rho) = 1-\rho, \quad (\rho * w_\eta)(t,x) = \int\limits_x^{x+\eta} w_\eta(y-x) \rho(t,y) \dy
\end{equation}
with different non-increasing, non-negative kernel functions $w_\eta$ on $[0,\eta]$ for some $\eta > 0$ such that $\int_0^\eta w_\eta(x)\dx=1$. Further, we consider a non-local sedimentation model given by
\begin{equation}\label{eq:sedimentation}
 g(\rho) = \rho (1-\rho)^\alpha, \quad v(\rho) = (1-\rho)^n, \quad (\rho * w_\eta)(t,x) = \int\limits_{-2\eta}^{2\eta} w_\eta(y) \rho(t,x+y) \dy,
\end{equation}
where $n\ge1$, $\alpha=0$ or $\alpha\ge1$, and $w_\eta$ is a symmetric, non-negative and piecewise smooth kernel function with support on $[-2\eta,2\eta]$, fulfilling $\int_\RR w_\eta(x)\dx=1$.

Solutions to \eqref{eq:nonloc}-\eqref{eq:initial} may contain complex but at the same time non-smooth components. This makes high-resolution numerical schemes desirable but also a very challenging task.
Recently, in~\cite{ChalonsGoatinVillada2018}, Discontinuous Galerkin (DG) schemes and Finite Volume WENO (FV-WENO) schemes have been proposed to obtain high-order approximations for a certain class of non-local conservation laws. The results of the DG schemes are very accurate but they require a very restrictive CFL condition. On the other hand, the presented FV-WENO schemes can be applied with larger time steps but they need an additional reconstruction step and are less accurate. 
Within this work, we propose several high-order central WENO (CWENO) schemes for the numerical solution of non-local conservation laws. We consider CWENO as a natural candidate for this problem class since the underlying reconstructions do not only provide single point values but a complete spatial reconstruction in every time step, which is beneficial to evaluate the integral terms. This way we want to achieve the same accuracy as the DG schemes of~\cite{ChalonsGoatinVillada2018} but without severe restrictions on the CFL condition and also without additional reconstruction steps as the FV-WENO schemes of~\cite{ChalonsGoatinVillada2018}. Further, by making use of the well-known linear scaling limiter of~\cite{ZhangShu2010,ZhangShu2011}, our CWENO schemes satisfy a maximum principle for suitable non-local conservation laws.

This work is organized as follows. In Section~\ref{sec:scheme}, we introduce a class of CWENO schemes for non-local conservation laws of the form~\eqref{eq:nonloc}-\eqref{eq:initial}. The requirements and the proof of a maximum principle are presented in Section~\ref{sec:maxPrinciple}. Section~\ref{sec:results} is devoted to numerical results confirming the stability and the convergence rates of the proposed schemes, and also the fulfillment of the maximum principle.

\section{CWENO schemes for non-local conservation laws}\label{sec:scheme}
The CWENO schemes considered in this work are based on a class of CWENO reconstructions recently presented in~\cite{CraveroPuppoSempliceVisconti2017}. In general, to solve \eqref{eq:nonloc}-\eqref{eq:initial}, we apply a spatial semi-discretization on a uniform grid with spatial mesh size $h$, grid points $x_j = x_0 + jh$ and corresponding finite volumes $I_j = [x_j - \hh, x_j + \hh] = [x_\jmh, x_\jph]$. Accordingly, integration over $I_j$ gives from~\eqref{eq:nonloc}
\begin{equation}\label{eq:semi-disc-exact}
 \frac{d}{dt} \brho(t,x_j) = - \frac1h \left( f(\rho)(t,x_\jph) - f(\rho)(t,x_\jmh)\right),
\end{equation}
an ordinary differential equation for the evolution of the cell averages
\begin{equation}
 \brho(t,x_j) = \frac1h \int\limits_{I_j} \rho(t,x) \dx,
\end{equation}
and from~\eqref{eq:initial} averaged initial conditions
\begin{equation}\label{eq:initialCellAverages}
 \brho(0,x_j) = \frac1h \int\limits_{I_j} \rho_0(x) \dx.
\end{equation}
We will denote the semi-discrete numerical approximation of those cell averages by $\brho_j(t)$, i.e.,
\begin{equation}
 \brho_j(t) \approx \brho(t,x_j).
\end{equation}
To evaluate an approximation of the right-hand side of~\eqref{eq:semi-disc-exact}, we will make use of CWENO reconstructions based on the cell averages $\brho_j(t)$, which is described in the following section. Afterwards, we explain all further ingredients to achieve a fully discrete scheme, beginning with the choice of a suitable quadrature rule to approximate the initial cell averages~\eqref{eq:initialCellAverages} and to evaluate the convolution term contained in~\eqref{eq:semi-disc-exact}.

\subsection{CWENO reconstruction procedure}

For the considered spatial semi-discretization, i.e., to approximate the convolution terms and to evaluate the fluxes at the cell interfaces in~\eqref{eq:semi-disc-exact}, we require a (high-order) reconstruction technique based on cell averages. Since the presented procedure is independent of the time variable, we consider $\rho = \rho(x)$ as function of the spatial variable only within this section to improve readability. Thus, based on cell averages $\bar\rho_j$ over all $I_j$,
\begin{equation}
  \brho_j = \frac1h \int\limits_{I_j} \rho(x) \dx,
\end{equation}
the aim is to reconstruct the underlying function $\rho$ by a piecewise polynomial approximation $\P$.
Within each interval $I_j$, $\P$ is built as convex combination of several polynomials. According to~\cite{CraveroPuppoSempliceVisconti2017}, for a CWENO reconstruction of order $2g+1$ ($g\in\{1,2,3\}$ within this work), we consider in each interval a polynomial $\P_\opt$ of degree $G=2g$ and $m=g+1$ polynomials $\P_1$, \ldots, $\P_m$ of degree $g$. For a fixed cell $I_j$ those fulfill the following properties:
\begin{equation}
 \frac1h \int\limits_{I_{j+l}} \P_\opt(x) \dx = \brho_{j+l} \qquad \forall l\in\{ -g,\ldots,g \}
\end{equation}
and for $k\in\{1,\ldots,m\}$
\begin{equation}
 \frac1h \int\limits_{I_{j-g+k-1+l}} \P_k(x) \dx = \brho_{j-g+k-1+l} \qquad \forall l\in\{ 0,\ldots,g \}.
\end{equation}
For the third-order method CWENO3 ($g=1$) $\P_\opt$ is the unique parabola conserving the three cell averages $\brho_\jmo$, $\brho_j$ and $\brho_\jpo$ over $I_\jmo$, $I_j$ and $I_\jpo$, respectively. $\P_1$ and $\P_2$ are one-sided linear reconstructions conserving $\brho_j$ and $\brho_\jmo$ respectively $\brho_\jpo$. In general, if the underlying data is smooth enough, $\P_\opt$ is a reconstruction of order $2g+1$ and the polynomials $\P_1$, \ldots, $\P_m$ are $(g+1)$th order accurate.

Next, one has to choose positive real coefficients $\c_0,\ldots,\c_m \in (0,1)$ such that $\sum_{k=0}^m \c_k = 1$ and we build
\begin{equation}
 \P_0(x) = \frac{1}{\c_0} \left( \P_\opt(x) - \sum\limits_{k=1}^m \c_k \P_k(x) \right).
\end{equation}

Finally, the reconstruction in $I_j$ is given by
\begin{equation}
 P(x) = \sum\limits_{k=0}^m \w_k \P_k(x)
\end{equation}
with weights
\begin{equation}
 \w_k = \frac{\alpha_k}{\sum\limits_{i=0}^m \alpha_i} \quad \text{with} \quad \alpha_k = \frac{\c_k}{(\S(\P_k) + \varepsilon)^p}
\end{equation}
and suitable smoothness indicator $\S(\P_k)$, where we apply the one of Jiang and Shu \cite{JiangShu1996},
\begin{equation}
 \S(\P_k) = \sum\limits_{l=1}^G h^{2l-1} \int\limits_{I_j} \big( \P_k^{(l)}(x) \big)^2 \dx \qquad k\in\{0,\ldots,m \}. 
\end{equation}
Note that explicit expressions for $S(\P_k)$ based on the given cell averages $\brho_j$ are available.
Further, based on the results in~\cite{Kolb2014,Kolb2016,CraveroPuppoSempliceVisconti2017}, one should use $\varepsilon=\varepsilon(h)=h^q$ with $q\in\{1,2\}$ and $p\ge2$. Actually, we apply $p=2$ and $q=2$ throughout this work and choose the coefficients $\c_k$ according to~\cite{CraveroPuppoSempliceVisconti2017} with $\c_0 = 0.5$.

\subsection{Choice of a suitable quadrature rule}

Another important ingredient for our numerical scheme is a suitable quadrature rule for the initial cell averages~\eqref{eq:initialCellAverages} and the approximation of the convolution terms in~\eqref{eq:semi-disc-exact}. We denote by $\nqn$ the number of quadrature nodes $\qvar_\qn \in [0,1]$ ($\qn\in\{1,\ldots,\nqn\}$) with corresponding weights $\qwn>0$ such that $\sum_{\qn=1}^\nqn \qwn = 1$. To obtain the desired order of accuracy for the entire scheme, we make use of quadrature rules of at least the same order of accuracy as the applied CWENO reconstruction procedure. In particular, we require that the applied quadrature rule is exact for the polynomials of the reconstruction step, which will become important for the maximum principle in Section~\ref{sec:maxPrinciple}. 
Additionally, we later require that one quadrature node lies on the right-hand boundary of the interval. This makes Radau-Legendre quadrature rules a suitable candidate since they achieve the maximum degree of exactness ($2\nqn-2$) under this condition. Without loss of generality we assume $\qvar_\nqn = 1$.

To achieve the desired orders of accuracy, we apply the Radau-Legendre quadrature rules with $R=2,3,4$ for CWENO3 ($g=1$), CWENO5 ($g=2$) and CWENO7 ($g=3$), respectively.
Note that without regard to the maximum principle, which comes with a more restrictive CFL condition, one could as well apply any other quadrature rule of sufficient accuracy to approximate the initial cell averages and to evaluate the convolution term - like the Gauß-Legendre formulas as used in~\cite{ChalonsGoatinVillada2018}.

Based on a given quadrature rule, we directly get an approximation of the initial cell averages
\begin{equation}\label{eq:initialCondDisc}
 \brho(0,x_j) = \frac1h \int\limits_{I_j} \rho_0(x) \dx \approx \sum\limits_{\qn=1}^\nqn \qwn \rho_0(x_\jmh + h \qvar_\qn) = \brho_j(0).
\end{equation}

\subsection{Computation of the convolution term}

Next, we consider the approximation of the convolution term contained in~\eqref{eq:semi-disc-exact}. Therefore, we apply the recently proposed Godunov/upwind type numerical flux function of~\cite{FriedrichKolbGoettlich2018} to achieve the semi-discrete system
\begin{equation}\label{eq:semiDiscreteSystem}
 \frac{d}{dt} \brho_j(t) = - \frac1h \left( V_\jph(t) g(\rho_\jph^-(t)) - V_\jmh(t) g(\rho_\jmh^-(t)) \right).
\end{equation}
Note that one could alternatively apply the (non-local) Lax-Friedrichs type numerical flux function as used in~\cite{ChalonsGoatinVillada2018} - in particular, in case that the requirements for the Godunov type flux function (as $v\ge0$) are not fulfilled.
In~\eqref{eq:semiDiscreteSystem}, all required density values are taken from the CWENO reconstruction of the respective time level.
In particular,
\begin{equation}
 \rho_\jmh^-(t) = \P_\jmo(t,x_\jmh) \quad \text{and} \quad \rho_\jph^-(t) = \P_j(t,x_\jph),
\end{equation}
where $\P_\jmo$ and $\P_j$ are the reconstructed polynomials in the intervals $I_\jmo$ and $I_j$ based on the cell averages at time $t$. The evaluation of $V_\jpmh(t)$ makes use of the above mentioned quadrature rule and reconstructed polynomials in a sufficiently large neighborhood of $I_j$.

For the non-local traffic flow model, with $Nh = \eta$ ($N \in \NN$), the approximate convolution term is given by 
\begin{equation}\label{eq:VjphTraffic}
V_{j+\frac{1}{2}}(t) = v\left(h\sum_{k=0}^{N-1}\sum_{\qn=1}^\nqn \qwn w_\eta^{\qn,k} \rho_{j+k+1}^{\qn}(t) \right)
\end{equation}
with
\begin{equation}
 \rho_{l}^{\qn}(t) = \P_l(t,x_{l-\half} + \qvar_\qn h)
\end{equation}
and
\begin{align}
 w_\eta^{\qn,k} = w_\eta\big((k + \qvar_\qn) h \big). 
\end{align}

Analogously, we have for the non-local sedimentation model
\begin{equation}\label{eq:VjphSedimentation}
V_{j+\frac{1}{2}}(t) = v\left(h\sum_{k=-N}^{N-1}\sum_{\qn=1}^\nqn \qwn w_\eta^{\qn,k} \rho_{j+k+1}^{\qn}(t) \right)
\end{equation}
with $2Nh = 4\eta$. 

\begin{remark}
 Note that the choice $Nh=\eta$ for the traffic flow model and $2Nh=4\eta$ for the sedimentation model is only applied here to simplify notation. To keep the formal order of accuracy for $\eta/h \not\in \mathbb{N}$ (or $2\eta/h \not\in \mathbb{N}$), one has to apply the provided formulas with $N=\lfloor \eta/h\rfloor$ (or $N=\lfloor 2\eta/h \rfloor$) and further use the underlying quadrature rule within the resulting subintervals $[x_{j+\frac{1}{2}}+Nh,x_{j+\frac{1}{2}}+\eta]$ at the right-hand boundary for the traffic flow model and at both boundaries $[x_{j+\frac{1}{2}}-\eta,x_{j+\frac{1}{2}}-Nh]$ and $[x_{j+\frac{1}{2}}+Nh,x_{j+\frac{1}{2}}+\eta]$ for the sedimentation model. Beside the convergence rate also all later results concerning the maximum principle stay valid this way.
\end{remark}

\begin{remark}
 In the case that the applied quadrature rule is not exact for the considered kernel function $w_\eta$, we replace $w_\eta^{\qn,k}$ by
 \begin{equation}
  \tilde w_\eta^{\qn,k} = \frac{w_\eta^{\qn,k}}{h \sum\limits_k \sum\limits_\qn \qwn w_\eta^{\qn,k}}
 \end{equation}
 so that
 \begin{equation}
  h \sum\limits_k \sum\limits_\qn \qwn \tilde w_\eta^{\qn,k} = 1.
 \end{equation}
 This way, since the kernel function is non-negative and we assumed $\qwn > 0$ for the quadrature weights, we ensure that $v$ in \eqref{eq:VjphTraffic} and~\eqref{eq:VjphSedimentation} is evaluated at a convex combination of the reconstructed density values $\rho_l^\qn$, which can be controlled to lie in a desired range (see Section~\ref{sec:maxPrinciple}). Further, replacing $w_\eta^{\qn,k}$ by $\tilde w_\eta^{\qn,k}$ does not change the order of accuracy for the approximate ``mean density'' since
 \begin{equation}
  h \sum\limits_k \sum\limits_\qn \qwn w_\eta^{\qn,k} \approx \int\limits_0^\eta w_\eta(x) \dx = 1
 \end{equation}
 for the traffic flow model and analogously with $\pm 2\eta$ as integration bounds for the sedimentation model.
 In particular, $\tilde w_\eta^{\qn,k} = w_\eta^{\qn,k}$ if the applied quadrature rule is exact for the considered kernel function.
\end{remark}

\subsection{Time discretization}

Within the previous sections, we have described a complete spatial semi-discretization of the original problem~\eqref{eq:nonloc}-\eqref{eq:initial}, resulting in an ordinary differential equation~\eqref{eq:semiDiscreteSystem} for the approximate cell averages $\brho_j(t)$ together with initial conditions~\eqref{eq:initialCondDisc}. Thus it remains to specify an appropriate time integration scheme. For the third-order method CWENO3, we apply the well-known third-order TVD Runge-Kutta scheme of~\cite{GottliebShu1998} with CFL restrictions according to~\cite{FriedrichKolbGoettlich2018} and the analysis in Section~\ref{sec:maxPrinciple}. Unfortunately, no one-step TVD/SSP Runge-Kutta schemes (with solely non-negative coefficients) are available for orders larger than four~\cite{RuuthSpiteri2002}. Only a fourth-order method is possible, unless one can provide a suitable semi-discretization for the backwards-in-time problem of the original problem.

Therefore, we will apply non-TVD Runge-Kutta methods of orders five~\cite[§3.2.5]{Butcher2008} and seven~\cite[p.\ 196]{Butcher2008} below to demonstrate the convergence rates of the proposed CWENO schemes CWENO5 and CWENO7. 
Nevertheless, to guarantee the maximum principle as stated in the following section, SSP schemes must be applied. Here we use the mentioned third-order scheme of~\cite{GottliebShu1998} for CWENO3 and the two-step Runge-Kutta methods~\cite{ketcheson2011strong} of order five with four stages and of order seven with nine stages\footnote{For the implementation of these methods we used the coefficients given by \url{http://www.sspsite.org/msrk.html}.} for CWENO5 and CWENO7, respectively. For the initialization of the two-step Runge-Kutta methods, we compute the first step with the fourth-order SSP scheme of~\cite{SpiteriRuuth2002} (with five stages) and a suitably reduced step size.

\section{Maximum principle}\label{sec:maxPrinciple}
The aim of this section is to modify the CWENO schemes presented so far in such a way that, while keeping the high order of accuracy, the resulting numerical approximations fulfill a maximum principle. Certainly, such a property is only desirable if the analytical solution of the underlying problem fulfills a maximum principle as it is the case for the considered non-local traffic flow problem. The following analysis covers a more general case, namely under the assumptions
\begin{align}\label{eq:assumptionsMaxPrinciple}
g\in C^1([0,\rho_{\max}];\mathbb{R}^+) \qquad &  \text{with} \ \ g'\geq 0, \phantom{\int} \notag\\
v\in C^1([0,\rho_{\max}];\mathbb{R}^+) \qquad & \text{with} \ \ v'\leq 0,\\
w_\eta\in C^1([0,\eta];\mathbb{R}^+) \qquad & \text{with} \ \ w_\eta'\leq 0 \ \ \text{and} \ \int_0^\eta w_\eta(x) dx=1\notag.
\end{align}

Under these conditions, for $\rho_0 \in \BV(\mathbb{R};[0,\rho_{\max}])$, the following maximum principle holds (see e.g.~\cite{ChiarelloGoatin2017}):
\begin{equation}
 \inf_\mathbb{R}\{\rho_0\}\leq \rho(t,x) \leq \sup_\mathbb{R}\{\rho_0\} \quad \text{for a.e.}\ x \in \RR,\ t>0.
\end{equation}

To achieve the same property for the approximate cell averages computed by our CWENO schemes, two modifications are necessary. First, by applying the well-known linear scaling limiter of~\cite{ZhangShu2010,ZhangShu2011}, we achieve that the resulting modified CWENO reconstruction only delivers points within the desired range - without worsening the convergence rate. Secondly, also similar to the results in~\cite{ZhangShu2010,ZhangShu2011}, we have to adapt the CFL condition to ensure the maximum principle for the fully discrete scheme.

\subsection{Linear scaling limiter of Zhang and Shu}

As noted above, we apply the linear scaling limiter of~\cite{ZhangShu2010,ZhangShu2011} to force our CWENO reconstructions into the desired range without worsening the convergence rate. Therefore, we assume that the analytical solution as well as the approximate cell averages $\brho_j$ at a certain time lie in the interval $[\rho_m,\rho_M]$. Then, we modify any reconstructed polynomial $\P_j(x)$ in the following way:
\begin{equation}
 \tilde \P_j(x) = \brho_j + \theta ( \P_j(x) - \brho_j ) \quad \text{with} \quad \theta = \min\left\{ \left\vert \frac{\rho_M-\brho_j}{M_j - \brho_j} \right\vert, \left\vert \frac{\rho_m-\brho_j}{m_j - \brho_j} \right\vert \right\},
\end{equation}
where
\begin{equation}
 M_j = \max\limits_{x\in I_j} \P_j(x) \quad \text{and} \quad m_j = \min\limits_{x\in I_j} \P_j(x).
\end{equation}
Note that it is sufficient to determine $M_j$ and $m_j$ based on evaluations of $\P_j$ in all points that are actually used within the scheme, i.e., in our case, evaluations of $\P_j$ in the quadrature nodes and in $x_\jph$ (for the Godunov type numerical flux function).

\subsection{Main result}

Based on the modified CWENO reconstructions, we will next prove that approximate solutions of the semi-discretized problem~\eqref{eq:semiDiscreteSystem} computed by the first-order Euler forward scheme fulfill a maximum principle. Then, the same property directly follows for arbitrary SSP methods (with solely non-negative coefficients) for a suitably adapted CFL condition. Note that the proof of~\cite{ZhangShu2011} cannot be applied in the context of non-local problems since we do not deal with monotone schemes here.

\begin{theorem}\label{theorem:maxPrinciple}
Let the convolution term in~\eqref{eq:semiDiscreteSystem} be given by~\eqref{eq:VjphTraffic}, where the underlying quadrature rule is exact for polynomials of the applied CWENO reconstruction and $\qvar_\nqn=1$, and let $\brho_j^n \in [\rho_m,\rho_M]$  ($j\in\mathbb{Z}$) as well as $\rho_l^\qn \in [\rho_m,\rho_M]$ for all reconstructed densities at a certain time $t^n$. Further, we assume~\eqref{eq:assumptionsMaxPrinciple} and 
\begin{equation}
 h\sum_{k=0}^{N-1}\sum_{\qn=1}^\nqn \qwn w_\eta^{\qn,k} = 1.
\end{equation}
Then, all approximate cell averages $\brho_j^\npo$ ($j\in\mathbb{Z}$) at time $t^\npo = t^n + \tau$ computed by a single Euler forward step fulfill
\begin{align*}
\rho_m \leq \brho_j^\npo \leq \rho_M,
\end{align*}
if the following CFL condition holds:
\begin{align}
 \tau&\leq\frac{\qw_\nqn h}{\qw_\nqn h w_\eta(0)\Vert v'\Vert\Vert g\Vert+ \Vert v\Vert \Vert g'\Vert},\label{CFL}
\end{align}
where $\Vert.\Vert$ denotes the $L^\infty$ norm over the respective domain.
\end{theorem}

\begin{proof}
Before considering $\brho_j^{n+1}$, we show some general inequalities, where we leave out the time dependency for a better readability, except for the cell averages $\brho_j^n$ and $\brho_j^\npo$.

First, \eqref{eq:VjphTraffic} and the mean value theorem yield
\begin{align}
V_{j-\frac{1}{2}}-V_{j+\frac{1}{2}}&=v\left(h\sum_{k=0}^{N-1}\sum_{\qn=1}^\nqn \qwn w_\eta^{\qn,k}\rho_{j+k}^\qn\right)-v\left(h\sum_{k=0}^{N-1}\sum_{\qn=1}^\nqn \qwn w_\eta^{\qn,k}\rho_{j+k+1}^\qn\right)\nonumber\\
&=v'(\xi)h\left(\sum_{k=0}^{N-1}\sum_{\qn=1}^\nqn \qwn w_\eta^{\qn,k}\rho_{j+k}^\qn-\sum_{k=0}^{N-1}\sum_{\qn=1}^\nqn \qwn w_\eta^{\qn,k}\rho_{j+k+1}^\qn\right)\nonumber\\
&=-v'(\xi)h\left(-\sum_{\qn=1}^\nqn \qwn w_\eta^{\qn,0}\rho_{j}^\qn+ \sum_{k=1}^{N-1}\sum_{\qn=1}^\nqn \qwn (w_\eta^{\qn,k-1}-w_\eta^{\qn,k})\rho_{j+k}^\qn+\sum_{\qn=1}^\nqn \qwn w_\eta^{\qn,N-1}\rho_{j+N}^\qn\right).
\label{starteq}
\end{align}
Since $v'(\xi)\le0$, $\qwn > 0$, $w_\eta(x)\ge0$ and monotonically decreasing (and $\rho_l^\qn \le \rho_M$), we get
\begin{align}\label{eq:inequForUpperBound}
V_{j-\frac{1}{2}}-V_{j+\frac{1}{2}} &\leq -v'(\xi)h\left(-\sum_{\qn=1}^\nqn \qwn w_\eta^{\qn,0}\rho_{j}^\qn+ \sum_{k=1}^{N-1}\sum_{\qn=1}^\nqn \qwn (w_\eta^{\qn,k-1}-w_\eta^{\qn,k})\rho_M+\sum_{\qn=1}^\nqn \qwn w_\eta^{\qn,N-1}\rho_M\right)\nonumber\\
&= -v'(\xi)h\left(-\sum_{\qn=1}^\nqn \qwn w_\eta^{\qn,0}\rho_{j}^\qn+ \sum_{\qn=1}^{R}\qwn (w_\eta^{\qn,0}-w_\eta^{\qn,N-1})\rho_M+\sum_{\qn=1}^\nqn \qwn w_\eta^{\qn,N-1}\rho_M\right)\nonumber\\
&\leq \Vert v'\Vert h \sum_{\qn=1}^\nqn \qwn w_\eta^{\qn,0}(\rho_M-\rho_{j}^\qn).
\end{align}

Analogously, we obtain from \eqref{starteq} 
\begin{align}\label{eq:inequForLowerBound}
V_{j-\frac{1}{2}}-V_{j+\frac{1}{2}} &\geq \Vert v'\Vert h \sum_{\qn=1}^\nqn \qwn w_\eta^{\qn,0}(\rho_m-\rho_{j}^\qn).
\end{align}

By multiplying inequality \eqref{eq:inequForUpperBound} by $g(\rho_M)$, subtracting $V_{j+\frac{1}{2}} g(\rho_{j+\frac{1}{2}}^-)$ and applying the mean value theorem, we obtain
\begin{align}\label{eq:inequDV}
V_{j-\frac{1}{2}} g(\rho_M)-V_{j+\frac{1}{2}} g(\rho_{j+\frac{1}{2}}^-) &\leq \Vert v'\Vert \Vert g \Vert h \sum_{\qn=1}^\nqn \qwn w_\eta^{\qn,0}(\rho_M-\rho_{j}^\qn)+ V_{j+\frac{1}{2}} (g(\rho_M)-g(\rho_{j+\frac{1}{2}}^-))\notag\\
&\leq \Vert v'\Vert \Vert g \Vert h \sum_{\qn=1}^\nqn \qwn w_\eta^{\qn,0}(\rho_M-\rho_{j}^\qn)+ \Vert v\Vert \Vert g'\Vert(\rho_M-\rho_{j+\frac{1}{2}}^-)\notag\\
&= h \sum_{\qn=1}^{\nqn-1} \qwn w_\eta^{\qn,0} \Vert v'\Vert \Vert g \Vert (\rho_M-\rho_j^\qn) + \qw_\nqn \left( h w_\eta^{\nqn,0} \Vert v'\Vert \Vert g \Vert +\frac{\Vert v\Vert \Vert g'\Vert}{\qw_\nqn} \right) (\rho_M-\rho_j^\nqn).
\end{align}
In the final step we used that $\rho_{j+\frac{1}{2}}^- = \rho_j^\nqn$ (since $\qvar_\nqn = 1$).
Applying an Euler forward step to compute $\brho_j^\npo$ means
\begin{equation}
\brho_j^{n+1} = \brho_j^n - \lambda\left(V_\jph g(\rho_\jph^-) - V_\jmh g(\rho_\jmh^-)\right)
\end{equation}
with $\lambda = \frac{\tau}{h}$. Since $V_\jmh\ge0$ under the given assumptions, $\rho_\jmh^- \le \rho_M$ and $g'\ge0$, we get
\begin{equation}\label{eq:estRhojnpo}
\brho_j^{n+1} \le \brho_j^n + \lambda\left(V_{j-\frac{1}{2}}^n g(\rho_M)-V_{j+\frac{1}{2}}^n g(\rho_{j+\frac{1}{2}}^-)\right).
\end{equation}
Since we assumed a quadrature rule that is exact for the polynomials applied in the CWENO reconstruction procedure, we know
\begin{equation}
 \brho_j^n = \sum\limits_{\qn=1}^\nqn \qwn \rho_j^\qn
\end{equation}
and get from~\eqref{eq:estRhojnpo} together with~\eqref{eq:inequDV}
\begin{multline}
 \brho_j^\npo \le \sum\limits_{\qn=1}^{\nqn-1} \qwn \bigg( \rho_j^\qn + \underbrace{\lambda h w_\eta^{\qn,0} \Vert v'\Vert \Vert g \Vert}_{\in[0,1]} (\rho_M-\rho_j^\qn) \bigg)\\
 + \qw_\nqn \bigg( \rho_j^\nqn + \underbrace{\lambda \bigg( h w_\eta^{\nqn,0} \Vert v'\Vert \Vert g \Vert + \frac{\Vert v\Vert \Vert g'\Vert}{\qw_\nqn} \bigg)}_{\in[0,1]} (\rho_M - \rho_j^\nqn) \bigg).
\end{multline}
Now, under the CFL condition \eqref{CFL} (note the monotonicity of $w_\eta$), we achieve $\brho_j^\npo \le \rho_M$.

Analogously, we obtain from~\eqref{eq:inequForLowerBound}
\begin{multline}
V_{j-\frac{1}{2}}^n g(\rho_{m})-V_{j+\frac{1}{2}}^n g(\rho_{j+\frac{1}{2}}^-)\geq h \sum_{\qn=1}^{\nqn-1} \qwn w_\eta^{\qn,0} \Vert v'\Vert \Vert g \Vert (\rho_m-\rho_j^\qn)\\
+ \qw_\nqn \left( h w_\eta^{\nqn,0} \Vert v'\Vert \Vert g \Vert +\frac{\Vert v\Vert \Vert g'\Vert}{\qw_\nqn} \right) (\rho_m-\rho_j^\nqn)
\end{multline}
to show $\brho_j^{n+1} \geq \rho_{m}$, which completes the proof.
\end{proof}

Given Theorem~\ref{theorem:maxPrinciple} and a SSP method for the time integration, we may directly deduce a maximum principle for the complete fully discrete scheme, since the discretized initial conditions~\eqref{eq:initialCondDisc} fulfill the maximum principle and the applied SSP scheme can be rewritten as composition of Euler forward steps.

\begin{theorem}\label{theorem:maxPrincipleSSP}
 Let a SSP method with SSP constant $c_{\mathit{SSP}}$ be given and approximate initial conditions according to~\eqref{eq:initialCondDisc}.
 Then, under the assumptions of Theorem~\ref{theorem:maxPrinciple} together with the CFL condition
\begin{equation}\label{eq:CFLSSP}
 \tau \leq c_{\mathit{SSP}} \frac{\qw_\nqn h}{\qw_\nqn h w_\eta(0)\Vert v'\Vert\Vert g\Vert+ \Vert v\Vert \Vert g'\Vert},
\end{equation}
the approximate solutions based on the presented CWENO semi-discretizations with linear scaling limiter satisfy
\begin{equation}
 \rho_m = \inf_\RR\{\rho_0\} \leq \brho_j(t^n) \leq \rho_M = \sup_\RR\{\rho_0\} \qquad \forall j\in\mathbb{Z} 
\end{equation}
for all $t^n = n \tau$, $n\in\NN$.
\end{theorem}

\begin{remark}\label{remark:timeStep}
 Applying Radau-Legendre quadrature rules with $\nqn=2,3,4$ nodes for CWENO3, CWENO5 and CWENO7 yields $\qw_\nqn=\frac14, \frac19, \frac1{16}$, respectively. Thus, to ensure the maximum principle a priori comes with a certain price regarding the time step size restriction (factor 4, 9 and 16, respectively).
Of course, one can alternatively only apply the inexpensive linear scaling limiter to ensure high-order reconstructions within the given bounds and project the result of each evolution step in time back into the feasible domain, resulting in a violation of the conservation property, or rerun the evolution step with a smaller step size only in case of a (significant) violation of the maximum principle.
 
Further note that for other numerical flux functions than the applied Godunov type flux function, e.g.\ Lax-Friedrichs, the reconstructed polynomials are typically evaluated at both ends of each interval, which makes Gauß-Lobatto quadrature rules a reasonable choice (as in~\cite{ZhangShu2011} for hyperbolic conservation laws). The latter could also be applied together with the Godunov type numerical flux function and come with similar step size restrictions as the Radau-Legendre formulas, $\qw_\nqn = \frac1{\nqn (\nqn-1)}$ instead of $\qw_\nqn = \frac1{\nqn^2}$, but in our case we would have to choose a by one larger $\nqn$ to fulfill the accuracy requirement, leading to a stricter step size restriction.
\end{remark}

\section{Numerical results}\label{sec:results}
Within this section we consider four numerical test cases, where the first three directly refer to~\cite{ChalonsGoatinVillada2018}. We begin with a non-smooth example to test the stability as well as accuracy of the proposed CWENO schemes in the presence of discontinuities. Afterwards, we consider two smooth examples to check the formal convergence rates. Each one of the smooth scenarios is based on the non-local traffic flow model and the non-local sedimentation model. The fourth example is designed to show that the proposed CWENO schemes with linear scaling limiter fulfill the maximum principle while keeping the high convergence rates.

\subsection{Non-smooth test case}

This non-smooth test case (including the applied spatial mesh size $h$) is taken from~\cite{ChalonsGoatinVillada2018}: For the non-local traffic flow model~\eqref{eq:nonloc}-\eqref{eq:trafficFlow} we consider the initial conditions
\begin{equation}
 \rho_0(x) = \begin{cases}
              0.95 & \text{for} \ x\in[-0.5,0.4],\\
              0.05 & \text{otherwise},
             \end{cases}
\end{equation}
together with periodic boundary conditions and compute the numerical solution at time $T=0.1$ with $w_\eta(x) = 3(\eta^2-x^2)/(2\eta^3)$ for $\eta=0.1$. Figure~\ref{fig:nonsmooth} shows the numerical results for $h=1/800$ and $\tau \approx 0.9 h/(h w_\eta(0)+1)$ with CWENO3, CWENO5 and CWENO7 in comparison to a reference solution computed with CWENO7 and $h=1/3200$. All schemes accurately capture the discontinuities, where CWENO7 gives the sharpest result (see zoomed view of the dash-dotted region on the right-hand side of Figure~\ref{fig:nonsmooth}).

\begin{figure}[hbt]
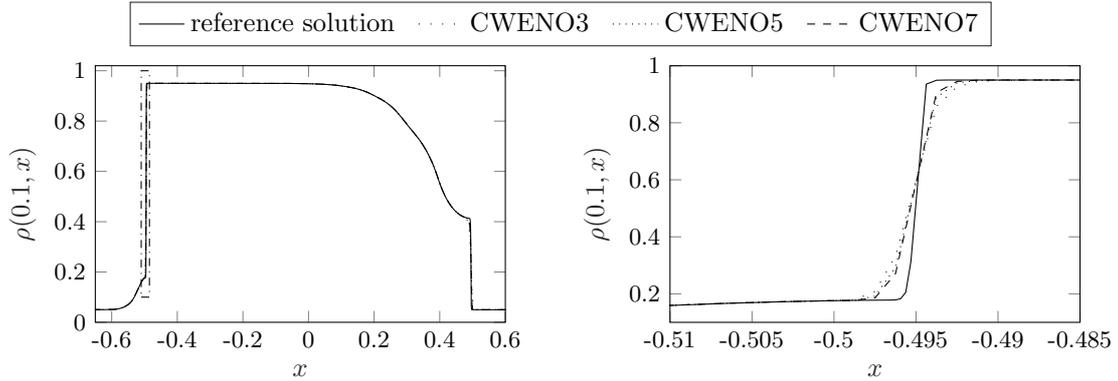

\centering
\begin{tikzpicture} 
    \begin{axis}[%
    hide axis,
    xmin=10,
    xmax=50,
    ymin=0,
    ymax=0.4,
    legend style={draw=white!15!black,legend cell align=left,legend columns= 4}
    ]
    \addlegendimage{black}
    \addlegendentry{reference solution \ \ };
    \addlegendimage{black, loosely dotted}
    \addlegendentry{CWENO3 \ \ };
        \addlegendimage{black, dotted}
    \addlegendentry{CWENO5 \ \ };
        \addlegendimage{black, dashed}
    \addlegendentry{CWENO7};
    \end{axis}
\end{tikzpicture}\\
\setlength{\fwidth}{0.38\textwidth}
\setlength{\xmin}{-0.65pt}
\setlength{\xmax}{0.6pt}
\setlength{\ymin}{0pt}
\setlength{\ymax}{1.02pt} 
\input{test1a}
\setlength{\xmin}{-0.51pt}
\setlength{\xmax}{-0.485pt}
\setlength{\ymin}{0.1pt}
\setlength{\ymax}{1.0pt}
\input{test1azoom}

 \caption{Numerical solution of the non-local traffic flow model with non-smooth initial conditions for CWENO3, CWENO5 and CWENO7 with $h=1/800$. Zoomed view of the dash-dotted region on the right.}
 \label{fig:nonsmooth}
\end{figure}

\subsection{Convergence tests}

As noted above, also the following two test cases are adapted from~\cite{ChalonsGoatinVillada2018} and are supposed to demonstrate that the formal order of accuracy is achieved for the proposed CWENO schemes.

\subsubsection{Non-local traffic flow model}

For the traffic flow model we consider the initial conditions
\begin{equation}
 \rho_0(x) = 0.5 + 0.4 \sin (\pi x)
\end{equation}
in the interval $[-1,1]$ (with periodic boundary conditions) and are interested in the numerical solution at time $T=0.15$. Here, we make use of different kernel functions,
\begin{align*}
w_\eta^1(x)=1/\eta, \qquad w_\eta^2(x)=2(\eta-x)/\eta^2, \qquad w_\eta^3(x)=3(\eta^2-x^2)/(2\eta^3),
\end{align*}
all with $\eta=0.2$. The mesh size $h$ is given by $h=1/20\cdot 2^{-n}$ with $n\in\{0,\ldots,5\}$. The time step size $\tau$ is given by $\tau\approx 0.9 h/(h w_\eta(0)+1)$. Tables~\ref{tab:convTest1L1} and~\ref{tab:convTest1Linfty} show the $L^1$ and $L^\infty$ errors of the corresponding numerical schemes in comparison to a reference solution computed with CWENO7 and $h=1/20\cdot 2^{-7}$. Obviously, CWENO3, CWENO5 and CWENO7 achieve their formal convergence rates (c.r.) for all considered kernels - until machine precision is reached.

\begin{table}[hbt]
\centering
\caption{$L^1$ errors and convergence rates for the smooth traffic flow example}
\label{tab:convTest1L1}
    \begin{tabular}{c c | c  c | c  c | c  c }
    &  &\multicolumn{2}{c}{CWENO3} & \multicolumn{2}{c}{CWENO5} & \multicolumn{2}{c}{CWENO7}\\
    \hline
& $n$& $L^1$ error & c.r.& $L^1$ error & c.r.& $L^1$ error & c.r.\\
    \hline\hline
\multirow{6}{*}{$w_\eta^1$}
&0&2.33e-04&-&2.58e-06&-&2.91e-07&-\\
&1&2.53e-05&3.20&9.35e-08&4.78&2.26e-09&7.01\\
&2&2.84e-06&3.15&3.10e-09&4.91&1.62e-11&7.12\\
&3&3.42e-07&3.06&9.95e-11&4.96&1.20e-13&7.07\\
&4&4.17e-08&3.04&3.14e-12&4.98&1.35e-15&6.48\\
&5&5.16e-09&3.01&9.89e-14&4.99&5.84e-16&1.21\\
\hline
\multirow{6}{*}{$w_\eta^2$}
&0&2.36e-04&-&2.99e-06&-&3.53e-07&-\\
&1&2.48e-05&3.25&1.11e-07&4.75&2.73e-09&7.01\\
&2&2.80e-06&3.15&3.67e-09&4.92&1.94e-11&7.13\\
&3&3.36e-07&3.06&1.17e-10&4.97&1.45e-13&7.07\\
&4&4.10e-08&3.03&3.70e-12&4.99&2.38e-15&5.93\\
&5&5.09e-09&3.01&1.16e-13&4.99&1.51e-15&0.65\\
\hline
\multirow{6}{*}{$w_\eta^3$}
&0&2.22e-04&-&2.86e-06&-&3.26e-07&-\\
&1&2.39e-05&3.22&1.06e-07&4.76&2.58e-09&6.98\\
&2&2.73e-06&3.13&3.50e-09&4.92&1.85e-11&7.12\\
&3&3.35e-07&3.03&1.12e-10&4.96&1.38e-13&7.07\\
&4&4.14e-08&3.02&3.54e-12&4.99&2.32e-15&5.89\\
&5&5.18e-09&3.00&1.12e-13&4.99&1.50e-15&0.63\\       
\hline\hline
    \end{tabular}    
\end{table}

\begin{table}[hbt]
\centering
\caption{$L^\infty$ errors and convergence rates for the smooth traffic flow example}
\label{tab:convTest1Linfty}
    \begin{tabular}{c c | c  c | c  c | c  c }
    &  &\multicolumn{2}{c}{CWENO3} & \multicolumn{2}{c}{CWENO5} & \multicolumn{2}{c}{CWENO7}\\
    \hline
& $n$& $L^\infty$ error & c.r.& $L^\infty$ error & c.r.& $L^\infty$ error & c.r.\\
    \hline\hline
\multirow{6}{*}{$w_\eta^1$}
&0&7.35e-04&-&1.13e-05&-&1.39e-06&-\\
&1&8.53e-05&3.11&4.09e-07&4.78&1.09e-08&6.99\\
&2&9.61e-06&3.15&1.32e-08&4.95&7.56e-11&7.17\\
&3&1.15e-06&3.06&4.22e-10&4.97&5.49e-13&7.11\\
&4&1.40e-07&3.04&1.33e-11&4.99&4.33e-15&6.99\\
&5&1.74e-08&3.01&4.16e-13&5.00&1.89e-15&1.20\\
\hline
\multirow{6}{*}{$w_\eta^2$}
&0&7.09e-04&-&1.39e-05&-&1.72e-06&-\\
&1&7.73e-05&3.20&5.05e-07&4.78&1.37e-08&6.97\\
&2&8.81e-06&3.13&1.64e-08&4.94&9.68e-11&7.15\\
&3&1.06e-06&3.06&5.21e-10&4.98&7.05e-13&7.10\\
&4&1.29e-07&3.03&1.64e-11&4.99&5.83e-15&6.92\\
&5&1.60e-08&3.01&5.14e-13&5.00&6.44e-15&-0.14\\
\hline
\multirow{6}{*}{$w_\eta^3$}
&0&6.99e-04&-&1.34e-05&-&1.61e-06&-\\
&1&7.73e-05&3.18&4.77e-07&4.81&1.27e-08&6.98\\
&2&8.77e-06&3.14&1.56e-08&4.94&9.07e-11&7.13\\
&3&1.05e-06&3.06&4.95e-10&4.98&6.61e-13&7.10\\
&4&1.29e-07&3.03&1.55e-11&4.99&5.55e-15&6.90\\
&5&1.59e-08&3.01&4.87e-13&5.00&5.77e-15&-0.06\\     
\hline\hline
    \end{tabular}    
\end{table}

\subsubsection{Non-local sedimentation model}

For the sedimentation model we consider
\begin{equation}
 v(\rho)=(1-\rho)^3, \qquad g(\rho)=\rho(1-\rho)
\end{equation}
and initial conditions
\begin{equation}
 \rho_0(x)=0.8\sin (\pi x)^{10}
\end{equation}
in the interval $[0,1]$ (with zero boundary conditions). The convolution kernel is given by a truncated parabola,
\begin{equation}
 w_\eta(x) = \tfrac1\eta K(\tfrac{x}{\eta})
\end{equation}
with
\begin{equation}
 K(y) = \begin{cases}
         \frac38 \left( 1 - \frac{y^2}{4} \right) & \text{for} \ \vert y\vert \le 2,\\
         0 & \text{otherwise}.
        \end{cases}
\end{equation}
We are interested in the numerical solution at time $T=0.04$ for $\eta=0.05$.
The mesh size $h$ is given by $h=1/20\cdot 2^{-n}$ with $n\in\{0,\dots,5\}$. The time step size $\tau$ is given by $\tau\approx 0.9 h/(3\eta w_\eta(0)+1)$. Tables \ref{tab:convTestSediL1} and \ref{tab:convTestSediLinfty} show the $L^1$ and $L^\infty$ errors of the corresponding numerical schemes in comparison to a reference solution computed with CWENO7 and $h=1/20\cdot 2^{-7}$. Again, the formal convergence rates are achieved, even though rather late for CWENO7 here and the final precision due to machine accuracy seems to be lower.

\begin{table}[hbt]
\centering
\caption{$L^1$ errors and convergence rates for the sedimentation example}
\label{tab:convTestSediL1}
    \begin{tabular}{ c | c  c | c  c | c  c }
      &\multicolumn{2}{c}{CWENO3} & \multicolumn{2}{c}{CWENO5} & \multicolumn{2}{c}{CWENO7}\\
    \hline
 $n$& $L^1$ error & c.r.& $L^1$ error & c.r.& $L^1$ error & c.r.\\
    \hline\hline
    0&2.78e-03&-&1.01e-03&-&4.91e-04&-\\
1&5.17e-04&2.43&3.42e-05&4.89&1.55e-05&4.98\\
2&7.42e-05&2.80&1.10e-06&4.95&2.43e-07&6.00\\
3&7.93e-06&3.23&4.07e-08&4.76&2.31e-09&6.72\\
4&7.68e-07&3.37&1.35e-09&4.91&1.78e-11&7.02\\
5&7.66e-08&3.33&4.39e-11&4.94&1.03e-12&4.11\\
\hline\hline
    \end{tabular}    
\end{table}

\begin{table}[hbt]
\centering
\caption{$L^\infty$ errors and convergence rates for the sedimentation example}
\label{tab:convTestSediLinfty}
    \begin{tabular}{ c | c  c | c  c | c  c }
     &\multicolumn{2}{c}{CWENO3} & \multicolumn{2}{c}{CWENO5} & \multicolumn{2}{c}{CWENO7}\\
    \hline
 $n$& $L^\infty$ error & c.r.& $L^\infty$ error & c.r.& $L^\infty$ error & c.r.\\
    \hline\hline
    0&1.37e-02&-&6.21e-03&-&2.53e-03&-\\
1&3.85e-03&1.83&3.68e-04&4.08&2.02e-04&3.65\\
2&6.69e-04&2.53&1.52e-05&4.59&5.61e-06&5.17\\
3&7.11e-05&3.23&7.40e-07&4.36&5.61e-08&6.64\\
4&7.11e-06&3.32&2.56e-08&4.85&4.32e-10&7.02\\
5&7.73e-07&3.20&8.44e-10&4.92&2.41e-10&0.84\\
\hline\hline
    \end{tabular}    
\end{table}

\subsection{Maximum principle}

Within the final example, we want to demonstrate the maximum principle and that the accuracy of the CWENO schemes is maintained also when the limiter is active. Therefore, we consider again the non-local traffic flow model with $w_\eta(x) = 3(\eta^2-x^2)/(2\eta^3)$. The initial conditions are piecewise defined by
\begin{equation}
 \rho_0(x) = \begin{cases}
              0 & \text{for} \ 0 \le x \le 1/8,\\
              q_1(x) & \text{for} \ 1/8 \le x \le 3/8,\\
              1 & \text{for} \ 3/8 \le x \le 5/8,\\
              q_2(x) & \text{for} \ 5/8 \le x \le 7/8,\\
              0 & \text{for} \ 7/8 \le x \le 1.\\
             \end{cases}
\end{equation}

\begin{figure}[hbt]
\centering
\setlength{\fwidth}{0.5\textwidth}
\input{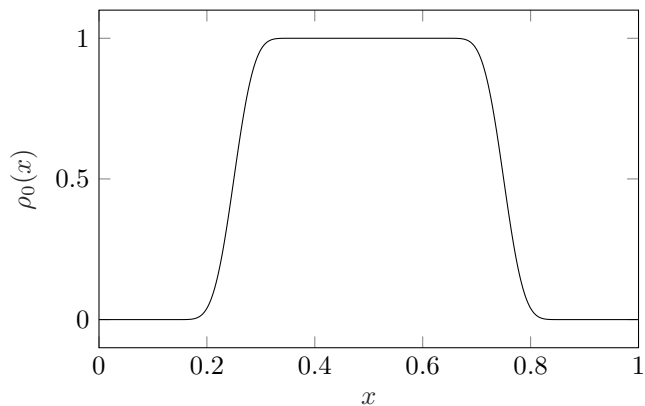}
\caption{Initial conditions $\rho_0(x)$ to test the maximum principle.}
\end{figure}

$q_1$ and $q_2$ are polynomials of degree $17$, chosen in such a way that the transitions are sufficiently smooth. We consider periodic boundary conditions and are interested in the numerical solution at time $T=0.05$ for $\eta=0.05$. The mesh size $h$ is given by $h=1/20\cdot 2^{-n}$ with $n\in\{0,\dots,6\}$. The time step size $\tau$ for CWENO3 in combination with the third order TVD Runge-Kutta method of~\cite{GottliebShu1998} is given by $\tau\approx 0.9 \cdot \frac14 h/(h w_\eta(0)/4+1)$. For CWENO5 and CWENO7 we apply the two-step Runge-Kutta methods of~\cite{ketcheson2011strong} with the corresponding orders with four and nine stages, respectively. 
The step sizes are then given by $\tau\approx 0.9 \cdot 0.21354 \cdot \frac19 h/(h w_\eta(0)/9+1)$ and $\tau\approx 0.9 \cdot 0.12444 \cdot \frac{1}{16} h/(h w_\eta(0)/16+1)$. 
The initial step of CWENO5 and CWENO7 is done by applying the five-stage fourth-order TVD Runge-Kutta method of~\cite{SpiteriRuuth2002} with $\tilde{\tau}\approx \min\{\tau,1.35 \cdot \frac19 h/(h w_\eta(0)/9+1)\cdot 2^{-5n/4}\} $ and $\tilde{\tau}\approx \min\{\tau,1.35 \cdot \frac{1}{16} h/(h w_\eta(0)/16+1)\cdot 2^{-7n/4}\}$, respectively.

Tables \ref{tab:convTestMaxL1} and \ref{tab:convTestMaxLinf} show the $L^1$ and $L^\infty$ errors of the corresponding numerical schemes in comparison to a reference solution computed with CWENO7 and $h=1/20\cdot 2^{-7}$. Due to the steep gradients, convergence rates close to the formal convergence rates are observed rather late here. 

\begin{table}[hbt]
\centering
\caption{$L^1$ errors and convergence rates for the maximum principle test}
\label{tab:convTestMaxL1}
    \begin{tabular}{ c | c  c | c  c | c  c }
      &\multicolumn{2}{c}{CWENO3} & \multicolumn{2}{c}{CWENO5} & \multicolumn{2}{c}{CWENO7}\\
    \hline
 $n$& $L^1$ error & c.r.& $L^1$ error & c.r.& $L^1$ error & c.r.\\
    \hline\hline

0&1.19e-02&-&8.65e-03&-&6.97e-03&-\\
1&5.71e-03&1.06&3.56e-03&1.28&2.69e-03&1.37\\
2&1.67e-03&1.78&5.65e-04&2.66&3.67e-04&2.87\\
3&4.02e-04&2.05&6.62e-05&3.09&1.91e-05&4.26\\
4&7.18e-05&2.49&3.45e-06&4.26&5.82e-07&5.04\\
5&7.89e-06&3.19&1.14e-07&4.92&5.28e-09&6.78\\
6&8.10e-07&3.28&3.76e-09&4.92&3.83e-11&7.11\\
\hline\hline
    \end{tabular}    
\end{table}

\begin{table}[hbt]
\centering
\caption{$L^\infty$ errors and convergence rates for the maximum principle test}
\label{tab:convTestMaxLinf}
    \begin{tabular}{ c | c  c | c  c | c  c }
     &\multicolumn{2}{c}{CWENO3} & \multicolumn{2}{c}{CWENO5} & \multicolumn{2}{c}{CWENO7}\\
    \hline
 $n$& $L^\infty$ error & c.r.& $L^\infty$ error & c.r.& $L^\infty$ error & c.r.\\
    \hline\hline

0&9.08e-02&-&8.73e-02&-&8.55e-02&-\\
1&7.75e-02&0.23&6.09e-02&0.52&5.21e-02&0.71\\
2&3.98e-02&0.96&1.65e-02&1.89&1.14e-02&2.19\\
3&1.64e-02&1.28&3.14e-03&2.39&1.28e-03&3.16\\
4&3.21e-03&2.35&1.79e-04&4.13&4.26e-05&4.91\\
5&3.94e-04&3.03&6.81e-06&4.72&5.02e-07&6.41\\
6&3.83e-05&3.36&2.42e-07&4.82&3.56e-09&7.14\\
\hline\hline
    \end{tabular} 
\end{table}

Figure~\ref{fig:maxprinciple} shows a comparison of the numerical solution of CWENO3 for $h=0.025$ once with and once without applying the linear scaling limiter. While slight violations of the bounds can be observed when the limiter is not applied, the density stays within $[0,1]$ when the limiter is used.
\begin{remark}
While the linear scaling limiter is very important to preserve the maximum principle, the additional restrictive factor $\gamma_R$ for the time step $\tau$ in Theorems~\ref{theorem:maxPrinciple} and~\ref{theorem:maxPrincipleSSP} seems to be relaxable. For both CWENO5 and CWENO7, the solution in the test case above stays within $[0,1]$ (up to machine precision) if this factor is neglected. For CWENO7 even (additionally) neglecting the SSP constant leads to similar results. Both observations give rise to apply rather aggressive strategies for the time stepping as mentioned in Remark~\ref{remark:timeStep}. Similarly, applying non-SSP Runge-Kutta methods in the final example only leads to very small violations of the maximum principle (as long as the linear scaling limiter is applied).
\end{remark}

\begin{figure}[hbt]
\centering
\begin{tikzpicture} 
    \begin{axis}[%
    hide axis,
    xmin=10,
    xmax=50,
    ymin=0,
    ymax=0.4,
    legend style={draw=white!15!black,legend cell align=left,legend columns= 4}
    ]
    \addlegendimage{black}
    \addlegendentry{reference solution \ \ };
    \addlegendimage{black, dashed}
    \addlegendentry{with limiter \ \ };
        \addlegendimage{black, dotted}
    \addlegendentry{without limiter \ \ };
        \addlegendimage{black, dashdotted}
    \addlegendentry{boundaries};
    \end{axis}
\end{tikzpicture}\\
\setlength{\fwidth}{0.35\textwidth}
\setlength{\xmin}{0pt}
\setlength{\xmax}{1pt}
\setlength{\ymin}{-0.05pt}
\setlength{\ymax}{1.05pt}
%
\begin{tikzpicture}

\begin{axis}[%
width=0.951\fwidth,
height=0.6\fwidth,
at={(0\fwidth,0\fwidth)},
scale only axis,
xmin=\xmin,
xmax=\xmax,
xlabel style={font=\color{white!15!black}},
xlabel={$x$},
ymin=\ymin,
ymax=\ymax,
ylabel style={font=\color{white!15!black}},
ylabel={$\rho(0.05,x)$},
axis background/.style={fill=white}
]
\addplot [color=black, forget plot]
  table[row sep=crcr]{%
0	-0\\
0.209375	0.00027818183424122\\
0.21328125	0.000802363097889058\\
0.21640625	0.00179874022333992\\
0.21875	0.00323670113470365\\
0.2203125	0.00475334437147379\\
0.221875	0.00694224034759605\\
0.2234375	0.0100831804026456\\
0.225	0.0145591792859787\\
0.2265625	0.0208826097968142\\
0.228125	0.0297168615484908\\
0.2296875	0.0418817294715312\\
0.23125	0.0583240408484049\\
0.2328125	0.0800312919682258\\
0.234375	0.107873200160257\\
0.23671875	0.16215579179168\\
0.2390625	0.230598704649644\\
0.24375	0.391634792055665\\
0.2484375	0.546335580414454\\
0.2515625	0.632387213755732\\
0.2546875	0.703244975878291\\
0.2578125	0.760587412482815\\
0.2609375	0.80673863822444\\
0.2640625	0.843897612018343\\
0.2671875	0.873902410021497\\
0.2703125	0.898215510031726\\
0.2734375	0.917979527804515\\
0.2765625	0.934084115075036\\
0.2796875	0.947224701192842\\
0.2828125	0.957949041648148\\
0.2859375	0.966692599833737\\
0.2890625	0.973805007718573\\
0.2921875	0.979569776218268\\
0.2953125	0.984219009079331\\
0.2984375	0.987944440889874\\
0.3015625	0.990905763095759\\
0.30546875	0.9937349128692\\
0.309375	0.995794572022979\\
0.31328125	0.997260859005243\\
0.31796875	0.99843695993795\\
0.32421875	0.999327607875041\\
0.33203125	0.999807710614336\\
0.34609375	0.999991036642234\\
0.44765625	0.999925088595217\\
0.49140625	0.999499437351189\\
0.5140625	0.998749016013915\\
0.53046875	0.997647043935033\\
0.54296875	0.996265218958489\\
0.55390625	0.994483545675228\\
0.56328125	0.992376594891443\\
0.571875	0.98983885633677\\
0.5796875	0.986910163294105\\
0.58671875	0.98366803690288\\
0.59375	0.979752102698065\\
0.6	0.975620819998495\\
0.60625	0.970796094055435\\
0.6125	0.96519787835623\\
0.61875	0.95874590975625\\
0.625	0.951360553359397\\
0.63125	0.942962956971311\\
0.6375	0.933475405502963\\
0.64375	0.92282295084538\\
0.65	0.910936675172344\\
0.65625	0.89775759946298\\
0.6625	0.883239562545208\\
0.66875	0.86735001636264\\
0.67578125	0.847810418192902\\
0.6828125	0.826497699918297\\
0.68984375	0.803414992377428\\
0.69765625	0.775715461515136\\
0.70546875	0.745907671248371\\
0.7140625	0.710785737232982\\
0.7234375	0.669863736072832\\
0.73359375	0.622752500326528\\
0.74453125	0.569203842027855\\
0.75703125	0.505063092126632\\
0.77265625	0.42168455363464\\
0.82578125	0.135265861023909\\
0.8328125	0.101213608308171\\
0.83828125	0.0772110670204007\\
0.84296875	0.058985507946604\\
0.846875	0.0457540061678665\\
0.85078125	0.0344533069603814\\
0.8546875	0.0251209349531811\\
0.8578125	0.0190293584567782\\
0.8609375	0.0140756898154095\\
0.8640625	0.0101497269781319\\
0.8671875	0.00712114249181339\\
0.8703125	0.00485036847641718\\
0.8734375	0.00319850277236222\\
0.87734375	0.00180683672745996\\
0.88203125	0.000835932175186516\\
0.88828125	0.000250264328813588\\
0.8984375	1.88333256603279e-05\\
0.95625	-0\\
1	-0\\
};
\addplot [color=black, dashed, forget plot]
  table[row sep=crcr]{%
0	0\\
0.175	2.34867398951621e-06\\
0.2	0.00514881288862612\\
0.225	0.08720711297234\\
0.25	0.513763558601893\\
0.275	0.909045478663762\\
0.3	0.986933822354289\\
0.325	0.999176525690553\\
0.35	0.999990992914201\\
0.475	0.999772961236622\\
0.5	0.999314655433536\\
0.525	0.998086307463236\\
0.55	0.995071757677994\\
0.575	0.988354580242635\\
0.6	0.974815108079359\\
0.625	0.950121682519233\\
0.65	0.909303049858424\\
0.675	0.848254039015601\\
0.7	0.765113571920161\\
0.725	0.661267620413087\\
0.75	0.539904343969883\\
0.775	0.408515386432242\\
0.8	0.267124358625273\\
0.825	0.131046519350334\\
0.85	0.050450017230097\\
0.875	0.0149233020081314\\
0.9	0.00157578754075893\\
0.925	-0\\
1	0\\
};
\addplot [color=black,dashdotted, forget plot]
  table[row sep=crcr]{%
0	1\\
1	1\\
};
\addplot [color=black,dashdotted, forget plot]
  table[row sep=crcr]{%
0	0\\
1	0\\
};
\addplot [color=black,dotted, forget plot]
  table[row sep=crcr]{%
0	3.46392521104466e-05\\
0.125	1.53266403988273e-05\\
0.15	0.000345262091693277\\
0.175	-0.00214944920211257\\
0.2	0.00526581803432724\\
0.225	0.0884558976911836\\
0.25	0.513972295133306\\
0.275	0.908919179123845\\
0.3	0.98702998317204\\
0.325	0.999398631073434\\
0.35	1.00003491671549\\
0.475	0.999786086298343\\
0.5	0.999328270659967\\
0.525	0.998087786153036\\
0.55	0.995038471062434\\
0.575	0.988276192175517\\
0.6	0.974724777223462\\
0.625	0.950104314979853\\
0.65	0.909373811256746\\
0.675	0.848259057138382\\
0.7	0.764879012922437\\
0.725	0.660873647686457\\
0.75	0.54025296888248\\
0.775	0.40871806584276\\
0.8	0.26687840897118\\
0.825	0.131001706098345\\
0.85	0.0507731647872509\\
0.875	0.0158740897204337\\
0.9	0.00221850336925744\\
0.925	-0.00105176732009449\\
1	3.46392521104466e-05\\
};
\end{axis}
\end{tikzpicture}%
\\
\setlength{\xmin}{0.16pt}
\setlength{\xmax}{0.2pt}
\setlength{\ymin}{-0.01pt}
\setlength{\ymax}{0.01pt}
%
\begin{tikzpicture}

\begin{axis}[%
width=0.951\fwidth,
height=0.6\fwidth,
at={(0\fwidth,0\fwidth)},
scale only axis,
xmin=\xmin,
xmax=\xmax,
xlabel style={font=\color{white!15!black}},
xlabel={$x$},
ymin=\ymin,
ymax=\ymax,
ylabel style={font=\color{white!15!black}},
ylabel={$\rho(0.05,x)$},
axis background/.style={fill=white}
]
\addplot [color=black, forget plot]
  table[row sep=crcr]{%
0	-0\\
0.209375	0.00027818183424122\\
0.21328125	0.000802363097889058\\
0.21640625	0.00179874022333992\\
0.21875	0.00323670113470365\\
0.2203125	0.00475334437147379\\
0.221875	0.00694224034759605\\
0.2234375	0.0100831804026456\\
0.225	0.0145591792859787\\
0.2265625	0.0208826097968142\\
0.228125	0.0297168615484908\\
0.2296875	0.0418817294715312\\
0.23125	0.0583240408484049\\
0.2328125	0.0800312919682258\\
0.234375	0.107873200160257\\
0.23671875	0.16215579179168\\
0.2390625	0.230598704649644\\
0.24375	0.391634792055665\\
0.2484375	0.546335580414454\\
0.2515625	0.632387213755732\\
0.2546875	0.703244975878291\\
0.2578125	0.760587412482815\\
0.2609375	0.80673863822444\\
0.2640625	0.843897612018343\\
0.2671875	0.873902410021497\\
0.2703125	0.898215510031726\\
0.2734375	0.917979527804515\\
0.2765625	0.934084115075036\\
0.2796875	0.947224701192842\\
0.2828125	0.957949041648148\\
0.2859375	0.966692599833737\\
0.2890625	0.973805007718573\\
0.2921875	0.979569776218268\\
0.2953125	0.984219009079331\\
0.2984375	0.987944440889874\\
0.3015625	0.990905763095759\\
0.30546875	0.9937349128692\\
0.309375	0.995794572022979\\
0.31328125	0.997260859005243\\
0.31796875	0.99843695993795\\
0.32421875	0.999327607875041\\
0.33203125	0.999807710614336\\
0.34609375	0.999991036642234\\
0.44765625	0.999925088595217\\
0.49140625	0.999499437351189\\
0.5140625	0.998749016013915\\
0.53046875	0.997647043935033\\
0.54296875	0.996265218958489\\
0.55390625	0.994483545675228\\
0.56328125	0.992376594891443\\
0.571875	0.98983885633677\\
0.5796875	0.986910163294105\\
0.58671875	0.98366803690288\\
0.59375	0.979752102698065\\
0.6	0.975620819998495\\
0.60625	0.970796094055435\\
0.6125	0.96519787835623\\
0.61875	0.95874590975625\\
0.625	0.951360553359397\\
0.63125	0.942962956971311\\
0.6375	0.933475405502963\\
0.64375	0.92282295084538\\
0.65	0.910936675172344\\
0.65625	0.89775759946298\\
0.6625	0.883239562545208\\
0.66875	0.86735001636264\\
0.67578125	0.847810418192902\\
0.6828125	0.826497699918297\\
0.68984375	0.803414992377428\\
0.69765625	0.775715461515136\\
0.70546875	0.745907671248371\\
0.7140625	0.710785737232982\\
0.7234375	0.669863736072832\\
0.73359375	0.622752500326528\\
0.74453125	0.569203842027855\\
0.75703125	0.505063092126632\\
0.77265625	0.42168455363464\\
0.82578125	0.135265861023909\\
0.8328125	0.101213608308171\\
0.83828125	0.0772110670204007\\
0.84296875	0.058985507946604\\
0.846875	0.0457540061678665\\
0.85078125	0.0344533069603814\\
0.8546875	0.0251209349531811\\
0.8578125	0.0190293584567782\\
0.8609375	0.0140756898154095\\
0.8640625	0.0101497269781319\\
0.8671875	0.00712114249181339\\
0.8703125	0.00485036847641718\\
0.8734375	0.00319850277236222\\
0.87734375	0.00180683672745996\\
0.88203125	0.000835932175186516\\
0.88828125	0.000250264328813588\\
0.8984375	1.88333256603279e-05\\
0.95625	-0\\
1	-0\\
};
\addplot [color=black, dashed, forget plot]
  table[row sep=crcr]{%
0	0\\
0.175	2.34867398951621e-06\\
0.2	0.00514881288862612\\
0.225	0.08720711297234\\
0.25	0.513763558601893\\
0.275	0.909045478663762\\
0.3	0.986933822354289\\
0.325	0.999176525690553\\
0.35	0.999990992914201\\
0.475	0.999772961236622\\
0.5	0.999314655433536\\
0.525	0.998086307463236\\
0.55	0.995071757677994\\
0.575	0.988354580242635\\
0.6	0.974815108079359\\
0.625	0.950121682519233\\
0.65	0.909303049858424\\
0.675	0.848254039015601\\
0.7	0.765113571920161\\
0.725	0.661267620413087\\
0.75	0.539904343969883\\
0.775	0.408515386432242\\
0.8	0.267124358625273\\
0.825	0.131046519350334\\
0.85	0.050450017230097\\
0.875	0.0149233020081314\\
0.9	0.00157578754075893\\
0.925	-0\\
1	0\\
};
\addplot [color=black,dashdotted, forget plot]
  table[row sep=crcr]{%
0	1\\
1	1\\
};
\addplot [color=black,dashdotted, forget plot]
  table[row sep=crcr]{%
0	0\\
1	0\\
};
\addplot [color=black,dotted, forget plot]
  table[row sep=crcr]{%
0	3.46392521104466e-05\\
0.125	1.53266403988273e-05\\
0.15	0.000345262091693277\\
0.175	-0.00214944920211257\\
0.2	0.00526581803432724\\
0.225	0.0884558976911836\\
0.25	0.513972295133306\\
0.275	0.908919179123845\\
0.3	0.98702998317204\\
0.325	0.999398631073434\\
0.35	1.00003491671549\\
0.475	0.999786086298343\\
0.5	0.999328270659967\\
0.525	0.998087786153036\\
0.55	0.995038471062434\\
0.575	0.988276192175517\\
0.6	0.974724777223462\\
0.625	0.950104314979853\\
0.65	0.909373811256746\\
0.675	0.848259057138382\\
0.7	0.764879012922437\\
0.725	0.660873647686457\\
0.75	0.54025296888248\\
0.775	0.40871806584276\\
0.8	0.26687840897118\\
0.825	0.131001706098345\\
0.85	0.0507731647872509\\
0.875	0.0158740897204337\\
0.9	0.00221850336925744\\
0.925	-0.00105176732009449\\
1	3.46392521104466e-05\\
};
\end{axis}
\end{tikzpicture}%
\setlength{\xmin}{0.89pt}
\setlength{\xmax}{0.93pt}
\setlength{\ymin}{-0.01pt}
\setlength{\ymax}{0.01pt}
%
\begin{tikzpicture}

\begin{axis}[%
width=0.951\fwidth,
height=0.6\fwidth,
at={(0\fwidth,0\fwidth)},
scale only axis,
xmin=\xmin,
xmax=\xmax,
xlabel style={font=\color{white!15!black}},
xlabel={$x$},
ymin=\ymin,
ymax=\ymax,
ylabel style={font=\color{white!15!black}},
ylabel={$\rho(0.05,x)$},
axis background/.style={fill=white}
]
\addplot [color=black, forget plot]
  table[row sep=crcr]{%
0	-0\\
0.209375	0.00027818183424122\\
0.21328125	0.000802363097889058\\
0.21640625	0.00179874022333992\\
0.21875	0.00323670113470365\\
0.2203125	0.00475334437147379\\
0.221875	0.00694224034759605\\
0.2234375	0.0100831804026456\\
0.225	0.0145591792859787\\
0.2265625	0.0208826097968142\\
0.228125	0.0297168615484908\\
0.2296875	0.0418817294715312\\
0.23125	0.0583240408484049\\
0.2328125	0.0800312919682258\\
0.234375	0.107873200160257\\
0.23671875	0.16215579179168\\
0.2390625	0.230598704649644\\
0.24375	0.391634792055665\\
0.2484375	0.546335580414454\\
0.2515625	0.632387213755732\\
0.2546875	0.703244975878291\\
0.2578125	0.760587412482815\\
0.2609375	0.80673863822444\\
0.2640625	0.843897612018343\\
0.2671875	0.873902410021497\\
0.2703125	0.898215510031726\\
0.2734375	0.917979527804515\\
0.2765625	0.934084115075036\\
0.2796875	0.947224701192842\\
0.2828125	0.957949041648148\\
0.2859375	0.966692599833737\\
0.2890625	0.973805007718573\\
0.2921875	0.979569776218268\\
0.2953125	0.984219009079331\\
0.2984375	0.987944440889874\\
0.3015625	0.990905763095759\\
0.30546875	0.9937349128692\\
0.309375	0.995794572022979\\
0.31328125	0.997260859005243\\
0.31796875	0.99843695993795\\
0.32421875	0.999327607875041\\
0.33203125	0.999807710614336\\
0.34609375	0.999991036642234\\
0.44765625	0.999925088595217\\
0.49140625	0.999499437351189\\
0.5140625	0.998749016013915\\
0.53046875	0.997647043935033\\
0.54296875	0.996265218958489\\
0.55390625	0.994483545675228\\
0.56328125	0.992376594891443\\
0.571875	0.98983885633677\\
0.5796875	0.986910163294105\\
0.58671875	0.98366803690288\\
0.59375	0.979752102698065\\
0.6	0.975620819998495\\
0.60625	0.970796094055435\\
0.6125	0.96519787835623\\
0.61875	0.95874590975625\\
0.625	0.951360553359397\\
0.63125	0.942962956971311\\
0.6375	0.933475405502963\\
0.64375	0.92282295084538\\
0.65	0.910936675172344\\
0.65625	0.89775759946298\\
0.6625	0.883239562545208\\
0.66875	0.86735001636264\\
0.67578125	0.847810418192902\\
0.6828125	0.826497699918297\\
0.68984375	0.803414992377428\\
0.69765625	0.775715461515136\\
0.70546875	0.745907671248371\\
0.7140625	0.710785737232982\\
0.7234375	0.669863736072832\\
0.73359375	0.622752500326528\\
0.74453125	0.569203842027855\\
0.75703125	0.505063092126632\\
0.77265625	0.42168455363464\\
0.82578125	0.135265861023909\\
0.8328125	0.101213608308171\\
0.83828125	0.0772110670204007\\
0.84296875	0.058985507946604\\
0.846875	0.0457540061678665\\
0.85078125	0.0344533069603814\\
0.8546875	0.0251209349531811\\
0.8578125	0.0190293584567782\\
0.8609375	0.0140756898154095\\
0.8640625	0.0101497269781319\\
0.8671875	0.00712114249181339\\
0.8703125	0.00485036847641718\\
0.8734375	0.00319850277236222\\
0.87734375	0.00180683672745996\\
0.88203125	0.000835932175186516\\
0.88828125	0.000250264328813588\\
0.8984375	1.88333256603279e-05\\
0.95625	-0\\
1	-0\\
};
\addplot [color=black, dashed, forget plot]
  table[row sep=crcr]{%
0	0\\
0.175	2.34867398951621e-06\\
0.2	0.00514881288862612\\
0.225	0.08720711297234\\
0.25	0.513763558601893\\
0.275	0.909045478663762\\
0.3	0.986933822354289\\
0.325	0.999176525690553\\
0.35	0.999990992914201\\
0.475	0.999772961236622\\
0.5	0.999314655433536\\
0.525	0.998086307463236\\
0.55	0.995071757677994\\
0.575	0.988354580242635\\
0.6	0.974815108079359\\
0.625	0.950121682519233\\
0.65	0.909303049858424\\
0.675	0.848254039015601\\
0.7	0.765113571920161\\
0.725	0.661267620413087\\
0.75	0.539904343969883\\
0.775	0.408515386432242\\
0.8	0.267124358625273\\
0.825	0.131046519350334\\
0.85	0.050450017230097\\
0.875	0.0149233020081314\\
0.9	0.00157578754075893\\
0.925	-0\\
1	0\\
};
\addplot [color=black,dashdotted, forget plot]
  table[row sep=crcr]{%
0	1\\
1	1\\
};
\addplot [color=black,dashdotted, forget plot]
  table[row sep=crcr]{%
0	0\\
1	0\\
};
\addplot [color=black,dotted, forget plot]
  table[row sep=crcr]{%
0	3.46392521104466e-05\\
0.125	1.53266403988273e-05\\
0.15	0.000345262091693277\\
0.175	-0.00214944920211257\\
0.2	0.00526581803432724\\
0.225	0.0884558976911836\\
0.25	0.513972295133306\\
0.275	0.908919179123845\\
0.3	0.98702998317204\\
0.325	0.999398631073434\\
0.35	1.00003491671549\\
0.475	0.999786086298343\\
0.5	0.999328270659967\\
0.525	0.998087786153036\\
0.55	0.995038471062434\\
0.575	0.988276192175517\\
0.6	0.974724777223462\\
0.625	0.950104314979853\\
0.65	0.909373811256746\\
0.675	0.848259057138382\\
0.7	0.764879012922437\\
0.725	0.660873647686457\\
0.75	0.54025296888248\\
0.775	0.40871806584276\\
0.8	0.26687840897118\\
0.825	0.131001706098345\\
0.85	0.0507731647872509\\
0.875	0.0158740897204337\\
0.9	0.00221850336925744\\
0.925	-0.00105176732009449\\
1	3.46392521104466e-05\\
};
\end{axis}
\end{tikzpicture}%
\caption{Numerical solution with and without using the linear scaling limiter for CWENO3 with $h=0.025$.}
\label{fig:maxprinciple}
\end{figure}
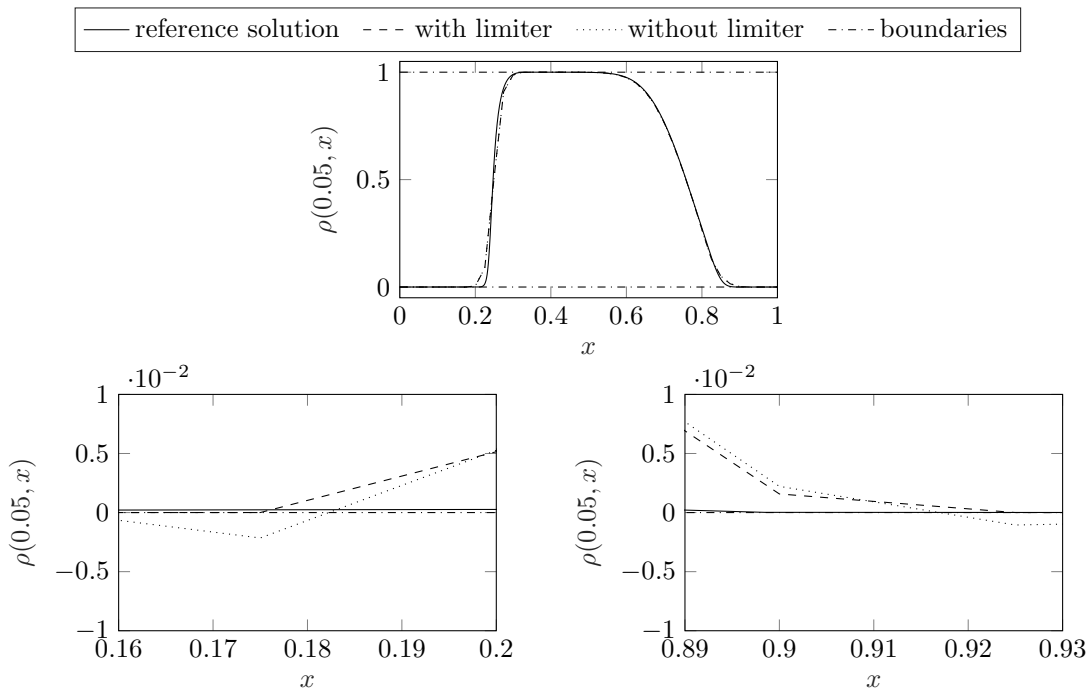

\section{Conclusion}\label{sec:conclusion}
Within this work we have presented high-order CWENO schemes for a class of non-local conservation laws. In contrast to the recently proposed Discontinuous Galerkin and Finite Volume WENO schemes, the proposed schemes neither require a very restrictive CFL condition nor an additional reconstruction step since the underlying CWENO reconstructions already provide stable high-order approximations for arbitrary points. By considering similar test cases as in~\cite{ChalonsGoatinVillada2018}, we demonstrated that the proposed CWENO schemes are able to accurately resolve solutions containing discontinuities as well as that they achieve their formal convergence rates for sufficiently smooth solutions.

Further, by making use of the well-known linear scaling limiter of~\cite{ZhangShu2010,ZhangShu2011}, we proved a maximum principle for the approximate solutions of suitable non-local conservation laws under a refined CFL condition. To guarantee the maximum principle, SSP time integration methods (with non-negative coefficients) must be applied. Within the class of Runge-Kutta methods, those are only available up to order four. For higher orders, we applied appropriate SSP two-step Runge-Kutta methods~\cite{ketcheson2011strong}. For suitably chosen time step sizes, the proposed CWENO schemes achieved high convergence rates within a test case where the rescaling due to the scaling limiter is permanently active.

Altogether, CWENO schemes seem to be a very promising class of schemes for non-local conservation laws. Within future work we intend to consider multi-dimensional problems as well as networks governed by non-local partial differential equations.

\bibliographystyle{siam}
\bibliography{CWENO_nonlocal}

\end{document}